\documentclass{amsart}
\usepackage{amssymb}
 \newtheorem{thm}{Theorem}[section]
 \newtheorem{cor}[thm]{Corollary}
 \newtheorem{lem}[thm]{Lemma}
 
 \theoremstyle{definition}
 
 \theoremstyle{remark}
 \newtheorem{rem}[thm]{Remark}
 
\begin{document}
\title[Semisimple Hopf algebras]{Structure theorems for semisimple
Hopf algebras of dimension $pq^3$}

\author[J. Dong]{Jingcheng Dong$^{a,b}$}
\address[a]{Department of Mathematics, Southeast University, Nanjing
211189, Jiangsu, People's Republic of China}
\address[b]{College of Engineering, Nanjing Agricultural University, Nanjing
210031, Jiangsu, People's Republic of China}

\email{dongjc@njau.edu.cn}



\begin{abstract}
Let $p,q$ be prime numbers with $p>q^3$, and $k$ an algebraically
closed field of characteristic $0$. In this paper, we obtain the
structure theorems for semisimple Hopf algebras of dimension $pq^3$.
\end{abstract}
\maketitle



\section{Introduction}\label{sec1}
Recently, various classification results were obtained for
finite dimensional semisimple Hopf algebras over an algebraically
closed field of characteristic $0$. Up to now, semisimple Hopf
algebras of dimension $p,p^2,p^3,pq,pq^2$ and $pqr$ have been
completely classified. See
\cite{Etingof,Etingof2,Masuoka1,Masuoka2,Zhu} for details.

In this paper, we study the structure of semisimple Hopf algebras of
dimension $pq^3$, where $p,q$ are prime numbers with $p>q^3$. We
prove that such Hopf algebras are either semisolvable in the sense
of \cite{Montgomery}, or isomorphic to a Radford biproduct $R\# A$
\cite{Radford}, where $A$ is a semisimple Hopf algebra of dimension
$q^3$, $R$ is a semisimple Yetter-Drinfeld Hopf algebra in
${}^{A}_{A}\mathcal{YD}$ of dimension $p$. In particular, we obtain
the structure theorem for semisimple Hopf algebras of dimension $8p$
for all prime numbers $p$.

Throughout this paper, all modules and comodules are left modules
and left comodules, and moreover they are finite-dimensional over an
algebraically closed field $k$ of characteristic $0$. ${\rm dim}$
means ${\rm dim}_k$. Our references for the theory of Hopf algebras
are \cite{Montgomery2} or \cite{Sweedler}. The notation for Hopf
algebras is standard. For example, the group of group-like elements
in $H$ is denoted by $G(H)$.

\section{Preliminaries}\label{sec2} Throughout this section, $H$
will be a semisimple Hopf algebra over $k$.

Let $V$ be an $H$-module. The character of $V$ is the element
$\chi=\chi_V\in H^*$ defined by $\langle\chi,h\rangle={\rm Tr}_V(h)$
for all $h\in H$. The degree of $\chi$ is defined to be the integer
${\rm deg}\chi=\chi(1)={\rm dim}V$. The antipode
$S$ induces an anti-algebra involution $*: R(H)\to R(H)$, given by
$\chi\to\chi^*:=S(\chi)$.

For any group-like element $g$ in $G(H^*)$, $m(g,\chi\chi^{*})>0$ if
and only if $m(g,\chi\chi^{*})= 1$ if and only if $g\chi=\chi$. The
set of such group-like elements forms a subgroup of $G(H^*)$.  See \cite[Theorem 10]{Nichols}.
Denote this subgroup by $G[\chi]$.

$H$ is said to be of type $(d_1,n_1;\cdots;d_s,n_s)$ as an algebra
if $d_1,d_2,\cdots,d_s$ are the dimensions of the simple $H$-modules
and  $n_i$ is the number of the non-isomorphic simple $H$-modules of
dimension $d_i$. That is, as an algebra,  $H$ is isomorphic to a
direct product of full matrix algebras $$H\cong k^{(n_1)}\times
\prod_{i=2}^{s}M_{d_i}(k)^{(n_i)}.$$

If $H^*$ is of type $(d_1,n_1;\cdots;d_s,n_s)$ as an algebra, then
$H$ is said to be of type $(d_1,n_1;\cdots;d_s,n_s)$ as a coalgebra.

\begin{lem}
Let $\chi$ be an irreducible character of $H$. Then

(1) The order of $G[\chi]$ divides $({\rm deg}\chi)^2$.

(2) The order of $G(H^*)$ divides $n({\rm deg}\chi)^2$, where $n$ is
the number of non-isomorphic irreducible characters of degree
deg$\chi$.
\end{lem}
\begin{proof}It follows from Nichols-Zoeller Theorem  \cite{Nichols2}.
See also \cite[Lemma 2.2.2]{Natale1}. \end{proof}

Let $\pi:H\to B$ be a Hopf algebra map and consider the subspace of
coinvariants
$$H^{co\pi}=\{h\in H|(id\otimes \pi)\Delta(h)=h\otimes 1\}.$$

Then $H^{co\pi}$ is a left coideal subalgebra of $H$. Moreover, we have
$${\rm dim}H ={\rm dim}H^{co\pi}{\rm dim}\pi(H) .$$

The left coideal subalgebra $H^{co\pi}$ is stable under the left
adjoint action of $H$. Moreover
if $H^{co\pi}$ is a Hopf subalgebra of $H$ then it is normal in $H$. See
\cite{Schneider} for more details.

Let $A$ be a semisimple Hopf algebra and let ${}^A_A\mathcal{YD}$
denote the braided category of Yetter-Drinfeld modules over $A$. Let
$R$ be a semisimple Yetter-Drinfeld Hopf algebra in
${}^A_A\mathcal{YD}$ \cite{Somm}.  As observed by D. E. Radford (see
\cite[Theorem 1]{Radford}), the Yetter-Drinfeld condition assures
that $R\otimes A$ becomes a Hopf algebra with additional structures.
This Hopf algebra is called the Radford biproduct of $R$ and $A$. We
denote this Hopf algebra by $ R\#A$.

\section{Semisimple Hopf algebras of dimension $pq^3$}\label{sec3}
\begin{lem}\label{lem1}
Let $H$ be a semisimple Hopf algebra of dimension $pq^3$, where
$p>q$ are prime numbers. If $H$ has a Hopf subalgebra $K$ of
dimension $pq^2$ then $H$ is lower semisolvable.
\end{lem}
\begin{proof}
Since the index $[H:K]=q$ is the smallest prime number dividing
${\rm dim}H$, the result in \cite{Kobayashi} shows that $K$ is a
normal Hopf subalgebra of $H$. Since the dimension of the quotient
$H/HK^+$ is $q$, the result in \cite{Zhu} shows that it is trivial.
That is, it is isomorphic to a group algebra or a dual group
algebra.

Since $K^*$ is also a semisimple Hopf algebra (see \cite{Larson}),
\cite[Lemma 2.2]{dong} and \cite[Theorem 5.4.1]{Natale3} show that
$K$ has a proper normal Hopf subalgebra $L$ of dimension $p,q,pq$ or
$q^2$. The results in \cite{Etingof,Masuoka1,Zhu} show that $L$ and
$K/KL^+$ are both trivial. Hence, we have a chain of Hopf
subalgebras $k\subseteq L\subseteq K\subseteq H$, which satisfies
the definition of lower semisolvability (see \cite{Montgomery}).
\end{proof}

In the rest of this section, $p,q$ will be distinct prime numbers
with $p>q^3$, and $H$ will be a semisimple Hopf algebra of dimension
$pq^3$.

Recall that a semisimple Hopf algebra $H$ is called
of Frobenius type if the dimensions of the simple $H$-modules divide
the dimension of $H$. Kaplansky conjectured that every
finite-dimensional semisimple Hopf algebra is of Frobenius type
\cite[Appendix 2]{Kaplansky}. It is still an open problem. Many
examples show that a positive answer to Kaplansky's conjecture would
be very helpful in the classification of semisimple Hopf algebras.
See \cite{dong2,Natale3} for example.

By \cite[Lemma 2.2]{dong}, $H$ is of Frobenius type and
$|G(H^*)|\neq1$. Therefore, the dimension of a simple $H$-module can only be
$1,q,q^2$ or $q^3$.  It follows that we have an equation
$pq^3=|G(H^*)|+aq^2+bq^4+cq^6$, where $a,b,c$ are the numbers of
non-isomorphic simple $H$-modules of dimension $q,q^2$ and $q^3$,
respectively. By Nichols-Zoeller Theorem \cite{Nichols2}, the order
of $G(H^*)$ divides ${\rm dim}H$. In particular, if $|G(H^*)|=pq^3$
then $H$ is a dual group algebra. We shall examine every possible
order of $G(H^*)$.

\begin{lem}
The order of $G(H^*)$ can not be $p,pq$ and $q$.
\end{lem}
\begin{proof}
From $pq^3=|G(H^*)|+aq^2+bq^4+cq^6$, we know that the order of $G(H^*)$
is divisible by $q^2$. The lemma then follows.
\end{proof}

\begin{lem}
If $|G(H^*)|=pq^2$ then $H$ is upper semisolvable.
\end{lem}
\begin{proof}
By \cite[Corollary 3.3]{Montgomery}, $H$ is upper semisolvable if
and only if $H^*$ is lower semisolvable. The lemma then follows from
Lemma \ref{lem1}.
\end{proof}

\begin{lem}
If $|G(H^*)|=q^2$ then $H$ is either semisolvable, or isomorphic to
a Radford biproduct $R\# A$, where $A$ is a semisimple Hopf algebra
of dimension $q^3$, $R$ is a semisimple Yetter-Drinfeld Hopf algebra
in ${}^{A}_{A}\mathcal{YD}$ of dimension $p$.
\end{lem}
\begin{proof}
From $pq^3=q^2+aq^2+bq^4+cq^6$, we have $a=q(p-bq-cq^3)-1$. Hence,
$a\neq 0$. The group $G(H^*)$, being abelian, acts by left
multiplication on the set $X_q$. The set $X_q$ is a union of orbits
which have length $1,q$ or $q^2$. Since $|X_q|=a$ does not divide
$p^2-1$, there exists at least one orbit with length $1$. That is,
there exists an irreducible character $\chi_q\in X_q$ such that
$G[\chi_q]=G(H^*)$. In addition, $G[\chi_q^*]=G(H^*)$ by \cite[Lemma
2.1.4]{Natale4}. This means that $g\chi_q=\chi_q=\chi_qg$ for all
$g\in G(H^*)$.

Let $C$ be a $q^2$-dimensional simple subcoalgebra of $H^*$,
corresponding to $\chi_q$. Then $gC=C=Cg$ for all $g\in G(H^*)$. By
\cite[Proposition 3.2.6]{Natale4}, $G(H^*)$ is normal in $k[C]$,
where $k[C]$ denotes the subalgebra generated by $C$. It is a Hopf
subalgebra of $H^*$ containing $G(H^*)$.  Counting dimension, we
know ${\rm dim}k[C]\geq 2q^2$. Since ${\rm dim}k[C]$ divides ${\rm
dim}H^*$ by Nichols-Zoeller Theorem \cite{Nichols2}, we know ${\rm
dim}k[C]=pq^3$, $pq^2$ or $q^3$.

If ${\rm dim}k[C]=pq^3$ then $k[C]=H^*$. Since $kG(H^*)$ is a group
algebra and the quotient $H^*/H^*(kG(H^*))^+$ is trivial (see
\cite{Etingof}), $H^*$ is lower semisolvable. Hence, $H$ is upper
semisolvable.

If  ${\rm dim}k[C]=pq^2$ then Lemma \ref{lem1} shows that $H^*$ is
lower semisolvable. Hence, $H$ is upper semisolvable.

In the rest of the proof, we consider the case that ${\rm
dim}k[C]=q^3$. In this case, $k[C]$ is of type $(1,q^2;q,q-1)$ as a
coalgebra. Considering the Hopf algebra map $\pi:H\to (k[C])^*$
obtained by transposing the inclusion $k[C]\subseteq H^*$, we have
that ${\rm dim}H^{co\pi}=p$. We shall examine every possible order of $G(H)$.

If $|G(H)| = pq^3$ then $H$ is a group algebra. If $|G(H)|=pq^2$ then $H$ is lower
semisolvable by Lemma \ref{lem1}.
If $|G(H)|=q^3$ then \cite[Lemma 4.1.9]{Natale4} shows that $H\cong
H^{co\pi}\#kG(H)$ is a biproduct.

If $|G(H)|=q^2$ then $H$ is lower semisolvable, or has a Hopf
subalgebra $K$ of dimension $q^3$, by the discussion above.  So, it
remains to consider the case that ${\rm dim}K=q^3$.
Notice that $K$ is of type $(1,q^2;q,q-1)$ as a
coalgebra by the discussion above.

If there exists an element $1\neq g\in G(H)$ such that $g$ appears in $H^{co\pi}$,
then $k\langle g\rangle\subseteq H^{co\pi}$ since $H^{co\pi}$ is a
subalgebra of $H$. But this contradicts \cite[Lemma 2.1]{dong} since
${\rm dim}k\langle g\rangle$ does not divide ${\rm dim}H^{co\pi}$.
Therefore, as a left coideal of $H$,
$$H^{co\pi}=k1\oplus\sum_iU_i\oplus\sum_jV_j\oplus\sum_kW_k,$$
where $U_i,V_j$ and $W_k$ are irreducible left coideal of $H$ of
dimension $q,q^2$ and $q^3$, respectively. On the one hand, ${\rm
dim}(K\cap H^{co\pi})=1+nq$, where $n$ is a non-negative integer.
On the other hand, ${\rm dim}K={\rm dim}(K\cap
H^{co\pi}){\rm dim}\pi(K)$ by \cite[Lemma 1.3.4]{Natale4}. Hence,
$n=0$ and $K\cap H^{co\pi}=k1$. By \cite[Theorem 3]{Radford}, $H\cong H^{co\pi}\#K$ is a biproduct.
This finishes the proof.
\end{proof}

\begin{lem}
If $|G(H^*)|=q^3$ then $H$ is either semisolvable, or isomorphic to
a Radford biproduct $R\# A$, where $A$ is a semisimple Hopf algebra
of dimension $q^3$, $R$ is a semisimple Yetter-Drinfeld Hopf algebra
in ${}^{A}_{A}\mathcal{YD}$ of dimension $p$.
\end{lem}
\begin{proof}
If $|G(H)|=q^3$ then the lemma follows from \cite[Lemma
4.1.8]{Natale4}. In all other cases, the lemma follows from lemmas
above.
\end{proof}

We are now in a position to give the main theorem.
\begin{thm}\label{thm1}
$H$ is either semisolvable, or isomorphic to a Radford biproduct
$R\# A$, where $A$ is a semisimple Hopf algebra of dimension $q^3$,
$R$ is a semisimple Yetter-Drinfeld Hopf algebra in
${}^{A}_{A}\mathcal{YD}$ of dimension $p$.
\end{thm}

\begin{rem}
The existence of semisimple Hopf algebra which is a biproduct as in Theorem \ref{thm1}
is still unknown. However, the theorem above shows that if $H$ is simple then $H$ is a biproduct.
But we do not know whether its converse is true. In fact, the only example which is simple as a
Hopf algebra and is also a biproduct appears in  \cite{Galindo}.

The semisimple Yetter-Drinfeld Hopf algebra $R$ in Theorem
\ref{thm1} heavily depends on the structure of the category
${}^{A}_{A}\mathcal{YD}$. In my opinion, the classification of such
Hopf algebras seems impossible at present. In fact, people are more
interested in semisimple Yetter-Drinfeld Hopf algebras over finite
groups, especially over those of prime order. See
\cite{Natale4,Somm} for example.
\end{rem}

Semisimple Hopf algebras of dimension $16$ are classified in
\cite{Kashina}. The structures of semisimple Hopf algebras of
dimension $24,40$ and $56$ are given in \cite{Natale4}. Therefore,
as an immediate consequence of Theorem \ref{thm1}, we have the
following corollary.
\begin{cor}
The structures of semisimple Hopf algebras of dimension $8p$ are
completely determined for all prime numbers $p$.
\end{cor}

\end{document}